\documentclass[11pt]{amsart}
\usepackage{amsmath,amssymb,amsthm,amstext,amsfonts}
\usepackage[latin1]{inputenc}
\usepackage[T1]{fontenc}
\usepackage{version,tabularx,multicol}
\usepackage{graphicx,float}
\usepackage{pstricks}

\theoremstyle{plain}
\newtheorem{theo}{Theorem}[section]

\newtheorem{prop}[theo]{Proposition}
\newtheorem{lem}[theo]{Lemma}
\newtheorem{defi}[theo]{Definition}
\theoremstyle{remark}
\newtheorem{rem}[theo]{Remark}

\newcommand{\ca}{{\mathcal A}}
\newcommand{\cb}{{\mathcal B}}

\newcommand{\ce}{{\mathcal E}}

\newcommand{\cf}{{\mathcal F}}

\newcommand{\cj}{{\mathcal J}}

\newcommand{\cl}{{\mathcal L}}
\newcommand{\cn}{{\mathcal N}}
\newcommand{\cm}{{\mathcal M}}

\newcommand{\cq}{{\mathcal Q}}
\newcommand{\cs}{{\mathcal S}}

\newcommand{\ct}{{\mathcal T}}

\newcommand{\E}{{\mathbb E}}

\newcommand{\M}{{\mathbb M}}
\newcommand{\N}{{\mathbb N}}
\renewcommand{\P}{{\mathbb P}}

\newcommand{\R}{{\mathbb R}}
\renewcommand{\S}{{\mathbb S}}

\newcommand{\ind}{{\bf 1}}

\newcommand{\Supp}{{\rm Supp}\;}

\newcommand{\expp}[1]{\mathop {\mathrm{e}^{ #1}}}

\begin{document}

\includeversion{commentaries}
%\excludeversion{commentaries}

\title[Dislocation measure of the fragmentation of a general Lévy tree]{Dislocation measure of the fragmentation of a general Lévy tree}

\date{\today}

%\author{Romain Abraham} 
%
%\address{
%Romain Abraham,
%MAPMO, CNRS UMR 6628,
%F\'ed\'eration Denis Poisson FR 2964,
%Université d'Orléans,
%B.P. 6759,
%45067 Orléans cedex 2
%FRANCE.
%}
%  
%\email{romain.abraham@univ-orleans.fr} 

\author{Guillaume Voisin}

\address{
Guillaume Voisin,
MAPMO CNRS UMR 6628, 
Fédération Denis Poisson FR 2964,
Université d'Orléans,
B.P. 6759,
45067 Orléans cedex 2
FRANCE.}

\email{guillaume.voisin@univ-orleans.fr}

\begin{abstract}
Given a general critical or sub-critical branching mechanism and its associated Lévy continuum random tree, we consider a pruning procedure on this tree using a Poisson snake.  It defines a fragmentation process on the tree. We compute the family of dislocation measures associated with this fragmentation. This work generalizes the work made for a Brownian tree \cite{as:psf} and for a tree without Brownian part \cite{ad:falp}.
\end{abstract}

\maketitle

\section{Introduction}

Continuous state branching processes (CSBP) have been introduced by Jirina \cite{j:sbpcss} and it is known since Lamperti \cite{l:lsbp} that these processes are the scaling limits of Galton-Watson processes. They model the evolution of a large population on a long time interval. The law of a CSBP is characterized by the so-called branching mechanism, which is the Laplace exponent of a spectrally positive Lévy process, ans is usually denoted by $\psi$. When the CSBP is critical or sub-critical, one can associate a continuum random tree (CRT) which describes the genealogy of the CSBP. Duquesne and Winkel \cite{dw:glt} has constructed genealogical trees associated with super-critical branching processes, we also cite Delmas \cite{d:hpsccsbp} for the construction of the height process when the branching process is super-critical. Aldous and Pitman \cite{ap:icrtebac,ap:sac} did a pioneering work in fragmentation processes involving discrete and continuum trees. The construction of fragmentation processes from CRTs have been studied by Abraham and Serlet \cite{as:psf} for the Brownian CRT (in the case where the Lévy measure of $\psi$ is null) and by Abraham and Delmas \cite{ad:falp} for the CRT without Brownian part (in the case where $\psi$ has no quadratic part). In these works, Lévy Poisson snakes are used to create marks on the CRT and to obtain a fragmentation process. In the first case, the marks are built on the skeleton of the CRT, in the second, they are placed on the nodes. Abraham, Delmas and Voisin \cite{adv:plcrt} constructed a general pruning of a CRT where the marks are placed on the whole CRT, skeleton and nodes. In this work, they study the law of the sub-tree obtained after the pruning according to the marks.

The aim of this article is to study the fragmentation process associated with a general CRT and more precisely the dislocation measure associated with this CRT. Note that this measure has been studied in the Brownian case and in the case without Brownian part (see \cite{as:psf} and \cite{ad:falp}).

The three following parts give a brief presentation of the mathematical objects and give the main results.
 
\subsection{The exploration process}
The coding of the CRT by its height process is well known. The height
process of Aldous' CRT \cite{a:crt3} is a normalized Brownian
excursion. In \cite{lglj:bplpep}, Le Gall and Le Jan associated  with
a Lévy process with no negative jumps that does not drift to infinity,
$X=(X_t,t\geq0)$, a CSBP and a Lévy CRT which keeps track of the
genealogy of the CSBP. Let $\psi$ be the Laplace exponent of the
process $X$. By the Lévy-Khintchine formula (and some additional
assumptions on $X$, see Section \ref{sec:levyprocess}), $\psi$ is such that $\E\left[ e^{-\lambda X_t}\right]=e^{t\psi(\lambda)}$ and can be expressed by 
$$\psi(\lambda)=\alpha \lambda + \beta \lambda^2 + \int_{(0,\infty)} (e^{-\lambda l}-1+\lambda l)\pi(dl)$$
with $\alpha\geq0$, $\beta\geq0$ and the Lévy measure $\pi$ is a
positive $\sigma$-finite measure on $(0,\infty)$ such that
$\int_{(0,\infty)}(l\wedge l^2)\pi(dl)<\infty$. Following
\cite{dlg:rtlpsbp}, we assume that $X$ is of infinite variation, which
implies that $\beta>0$ or $\int_{(0,1)}l\pi(dl)=+\infty$. The term
$\alpha$ is a drift term (if $\psi(\lambda)=\alpha\lambda$, $X$ is a
Cauchy process), $\beta$ is the quadratic term (if
$\psi(\lambda)=\beta\lambda^2$, $X$ is a Brownian motion) and $\pi$ gives the
law of the jumps of $X$.

We first construct the height process $H=(H_t,t\geq0)$ associated with
the process $X$ (see Section \ref{section:height}). This process codes
for a continuum random tree : each individual $t$ is at distance $H_t$
from the root and the last common ancestor of the individuals $s$ and
$t$ ($s<t$) is at distance:
$$H_{s,t} = \inf\{ H_u ; u\in [s,t]  \}$$
(see Section \ref{section:height} for a formal definition of a
continuum random tree and its coding by the height process).

This
height process is an important object but is not a Markov process in
general. Thus we introduce the exploration process
$\rho=(\rho_t,t\geq0)$ which is a càd-làg, strong Markov process
taking values in $\cm_f(\R_+)$, the set of finite measures on $\R_+$
endowed with the topology of weak convergence. It is defined by :
\begin{equation*}
\rho_t(dr)=\beta\ind_{[0,H_t]}(r)\; dr + \sum_{\stackrel{0<s\le t}
  {X_{s-}<I_t^s}}(I_t^s-X_{s-})\delta_{H_s}(dr)
\end{equation*}
where $\displaystyle I_t^s=\inf_{s\le u\le t}X_u$.\\
 The height process can easily be recovered from the exploration
 process as $H_t=H(\rho_t)$ where $H(\mu)$ is the supremum of the
 closed support of the measure $\mu$ (with the convention that
 $H(0)=0$). Informally, $\rho_t$ can be seen as a measure on the
 branch from the root to the individual $t$ which gives the intensity
 of the branching points (associated with individuals situated "on the
 right" of $t$) along that branch (see Bismut decomposition of Proposition
 \ref{Bismut} and the Poisson representation of the process of Lemma
 \ref{decomp poisson}). We can hence see that the regular part of the
 measure $\rho_t$ gives birth to binary branching points whereas the
 atoms of the measure (which correspond to jumps of the Lévy process
 $X$) lead to nodes of infinite index.

\subsection{The fragmentation}
A fragmentation process is a Markov process which describes how an  object with given total mass evolves as it breaks into several fragments randomly as time passes. This kind of processes has been widely studied in \cite{b:rfcp}. To be more precise, the state space of a fragmentation process is the space of non-increasing sequences of masses with finite total mass
$$\cs^{\downarrow}=\{  \mathbf{s} =(s_1,s_2,\dots); s_1\geq s_2 \geq \dots \geq 0 \; \mbox{and} \; \sum_{k=1}^\infty s_k <\infty \}.$$
We denote by $P_\mathbf{s}$ the law of a $\cs^{\downarrow}$-valued process $\Lambda=(\Lambda^\theta , \theta \geq0)$ starting at $\mathbf{s}=(s_1,s_2,\dots)\in\cs^{\downarrow}$. For fixed $\theta>0$, we write $(\Lambda^\theta)=(\Lambda^\theta_1,\Lambda^\theta_2,\dots)\in \cs^\downarrow$ and $\sum(\Lambda^\theta)=\sum_{i\geq1}\Lambda^\theta_i$ for the sum of the masses of the elements at time $\theta$. We say that $\Lambda$ is a fragmentation process if it is a Markov process such that $\theta\mapsto \sum(\Lambda^\theta)$ is decreasing and if it fulfills the fragmentation property : the law of $(\Lambda^\theta,\theta \geq 0)$ under $P_\mathbf{s}$ is the non-increasing reordering of the fragments of independent processes of respective laws $P_{(s_1,0,\dots)}$, $P_{(s_2,0,\dots)}$,$\dots$ In other words, each fragment behaves independently of the others, and its evolution depends only of its initial mass. Hence, it suffices to study the laws $P_r:=P_{(r,0,\dots)}$ that is the law of the fragmentation process starting with a single mass $r\in(0,\infty)$.

We want to construct a fragmentation process by cutting a Lévy CRT
into several subtrees. The lengths of the height processes that code
each subtrees, ranked in decreasing order, form an element of
$\cs^{\downarrow}$. In order to construct our fragmentation process,
we need to place marks on the CRT which give the different cut
points and the number of marks must increase as time passes.

There will be two sort of marks : some are lying on the nodes of
infinite index whereas the others are "uniformly" distributed on the
skeleton of the tree.

The nodes of the tree are marked independently and, at time $\theta$,
a node with mass $m$ is marked with probability $1-e^{-m\theta}$. To
have a consistent construction as $\theta$ varies, we use a
coupling construction so that the marks present at time $\theta$ are still marks at a
further time.

For the marks on the skeleton of the CRT, we use a Lévy Poisson snake
similar to those of \cite{dlg:rtlpsbp} but we must introduce the new
parameter $\theta$. At fixed time $\theta$, the marks on the lineage
of an individual $t$ will be distributed as a Poisson process with
intensity $2\beta\theta\ind_{[0,H_t]}(r)dr$, but the marks on two
common lineages must be the same and a coupling construction must also
apply.

By cutting according to these marks, we obtain a set of fragments. Let $s_1,s_2,\dots$ be the "sizes" of these fragments ranked by non-increasing order completed with 0 if necessary so that $(s_1,s_2,\dots)\in \cs^{\downarrow}$. When time $\theta$ increases, the number of marks increases and the fragments break again. Thus we obtain a process $(\Lambda^\theta,\theta\geq0)$, Theorem \ref{theofragmentation} checks that this process is a fragmentation process.

The choice of the parameters for the marks can be surprising as the
pruning of \cite{adv:plcrt} is much more general but the particular
pruning considered here leads
to a pruned exploration process that fulfills Lemma
\ref{lem:excursions} which is necessary for getting a fragmentation
process. We don't know if other pruning give such a property; one may
conjecture that it is the only one.

\subsection{The dislocation measure}\label{intro:section3}

The evolution of the process $\Lambda$ is described by a family $(\nu_r,r\geq0)$ of $\sigma$-finite measures called dislocation measures. $\nu_r$ describes how a fragment of size $r$ breaks into smaller fragments. In the case of self-similar fragmentations (with no loss of mass), the dislocation measure characterizes the law of the fragmentation process. In the general case, the characterization is an open problem.

To be more precise, we define $\ct=\{ \theta\geq0 ; \Lambda^\theta \neq \Lambda^{\theta-} \}$ the set of jumping times of the process $\Lambda$. The dislocation process of the CRT fragmentation $\sum_{\theta\in \ct}\delta_{\theta,\Lambda^\theta}$ is a point process with intensity $d\theta\, \tilde{\nu}_{\Lambda^{\theta-}}(d\mathbf{s})$, where $(\tilde{\nu}_{\mathbf{x}},\mathbf{x}\in \cs^\downarrow)$  is a family of $\sigma$-finite measure on $\cs^\downarrow$. There exists a family $(\nu_r,r>0)$ of $\sigma$-finite measures on $\cs^\downarrow$ such that for any $\mathbf{x}=(x_1,x_2,\dots)\in \cs^\downarrow$ and any non-negative function $F$, defined on $\cs^\downarrow$,

$$\int F(\mathbf{s}) \tilde{\nu}_{\mathbf{x}}(d\mathbf{s}) = \sum_{i\geq1,x_i>0} \int F(\mathbf{x}^{i,\mathbf{s}})\nu_{x_i}(d\mathbf{s})$$

where $\mathbf{x}^{i,\mathbf{s}}$ is the decreasing reordering of the merging of the sequences $\mathbf{s}$ and $\mathbf{x}$, where $x_i$ has been removed of the sequence of $\mathbf{x}$.

We will show in Section \ref{another} that the measure $\nu_r$ can be written as  
$$\nu_r = \nu^{nod}_r+\nu^{ske}_r$$
where $\nu^{nod}$ corresponds to a mark that appears on the node whereas $\nu^{ske}$ to a mark on the skeleton.

The expression of the measure $\nu_r^{ske}$ is the main result of this article :

\begin{theo}\label{theointro} 
Let $S$ be a subordinator with Laplace exponent $\psi^{-1}$, let $\pi_*$ be its Lévy measure.

\begin{enumerate}
\item For all non negative measurable function $F$ on $\cs^{\downarrow}$,
$$\int_{\R_+\times \cs^{\downarrow}} F( \mathbf{x} ) \nu_r^{nod}(d \mathbf{x} )\pi_*(dr) =\int \pi(dv) \E \left[  S_v F\left( (\Delta S_u , u\leq v) \right)\right]$$
where $(\Delta S_u , u\leq v)\in \cs^\downarrow$ represents the  jumps of $S$ before time $v$, ranked by decreasing order.

\item The measure $\nu_r^{ske}$ charges only the set of elements of $\cs^{\downarrow}$ of the form $(x_1,x_2,0,\ldots)$ with $x_1\geq x_2$ and $x_1+x_2=r$. It is the "distribution" of the non-increasing reordering of the lengths given by the measure $\hat{\nu}_r^{ske}$ defined by
$$\int_{\R_+\times \cs^{\downarrow}} \frac{1}{x_2} (1-e^{-\lambda_1 x_1})(1-e^{-\lambda_2 x_2}) \hat{\nu}_r^{ske}( d\mathbf{x} )\pi_*(dr) = 2\beta \psi^{-1}(\lambda_1)\psi^{-1}(\lambda_2).$$
\end{enumerate}
\end{theo}

\begin{rem}
Under $\hat{\nu}_r^{ske}(d\mathbf{x})\pi_*(dr)$, the lengths of the two fragments are "independent".
\end{rem}
\begin{rem}
We will see in Section \ref{another} that the measure $\nu^{nod}$ is the same as the measure $\nu$ in the case of a tree without Brownian part ($\beta=0$). Thus the proof of Part 1 of Theorem \ref{theointro} is the same as in \cite{ad:falp}. Only Part 2 needs a proof.
\end{rem}

\section {The Lévy snake : notations and properties}

\subsection{The Lévy process}\label{sec:levyprocess}
We consider a $\R$-valued Lévy process $(X_t, t\geq 0)$ with no negative jumps, starting from 0 characterized by its Laplace exponent $\psi$ given by 
\[
\psi(\lambda)=\alpha_0\lambda+\beta \lambda^2+\int_{(0,+\infty)}\pi(d\ell)
\left(\expp{-\lambda\ell}-1+\ind_{\ell < 1}\lambda\ell\right),  
\]
with $\beta\geq 0$ and the Lévy measure $\pi$ is a positive, $\sigma$-finite measure on $(0,+\infty )$ such that $\int_{(0,+\infty)} (1\wedge
\ell^2)\pi(d\ell)<\infty$. We also assume that $X$ 
\begin{itemize}
\item  has first moments (i.e. $\int_{(0,+\infty)} (\ell\wedge
\ell^2)\pi(d\ell)<\infty$),  
\item  is of infinite variation (i.e. $\beta>0$ or 
  $\int_{(0,1)}\ell\pi(d\ell)=+\infty$),
\item  does not drift to $+\infty$. 
\end{itemize}

With the first assumption, the Lévy exponent can be written as
\[
\psi(\lambda)=\alpha\lambda+\beta \lambda^2+\int_{(0,+\infty)}\pi(d\ell)
\left(\expp{-\lambda\ell}-1+\lambda\ell\right),  
\]
with $\alpha\geq 0$ thanks to the third assumption.

Let $\cj=\{t\geq 0; X_t\neq X_{t-}\}$ be the set of jumping times of the process $X$.\\
\\
For $\lambda\geq \frac{1}{\epsilon}>0$, we have $e^{-\lambda l}-1+\lambda l \geq \frac{1}{2}\lambda l \ind_{l\geq 2\epsilon}$ this implies that $\lambda^{-1}\psi(\lambda) \geq \alpha + \beta \frac{1}{\epsilon} +\int_{(2\epsilon , \infty )} l\pi(dl)$. We deduce that
$$\lim_{\lambda \rightarrow \infty} \frac{\lambda}{\psi(\lambda)} = 0.$$

Let $I=(I_t,t\ge 0)$ be the infimum process of $X$, $I_t=\inf_{0\le  s\le t}X_s$.
We also denote for all $0\leq s\leq t$, the minimum of $X$ on $[s,t]$ :
\[
I_t^s=\inf_{s\le r\le t}X_r.
\]
The point 0 is regular for the Markov process $X-I$, and $-I$ is the local time of $X-I$ at 0 (see \cite{b:pl}, Chap. VII). Let $\N$ be the excursion measure of the process $X-I$ away from 0, and let $\sigma=\inf\{ t>0 ; X_t-I_t=0 \}$ be the lengths of the generic excursion of $X-I$ under $\N$. Notice that, under $\N$, $X_0=I_0=0$.\\

Thanks to \cite{b:pl}, Theorem VII.1, the right-continuous inverse  of the process $-I$ is a subordinator with Laplace exponent $\psi^{-1}$. We have already seen that this exponent has no drift, because $\lim_{\lambda \rightarrow \infty} \lambda\psi(\lambda)^{-1} = 0$. We denote by $\pi_*$ its Lévy measure : for all $\lambda\geq0$
$$\psi^{-1}(\lambda) = \int_{(0,\infty)} \pi_*(dl) (1-e^{\lambda l}).$$
Under $\N$, $\pi_*$ is the "law" of the length of the excursions, $\sigma$. By decomposing the measure $\N$ w.r.t. the distribution of $\sigma$, we get that $\N(d\ce) = \int_{(0,\infty)} \pi_*(dr) \N_r(d\ce)$, where $(\N_r, r\in(0,\infty))$ is a measurable family of probability measures on the set of excursions (that is to say for all $A$, $r\mapsto \N_r(A)$ is $\cb(\R_+)$-measurable) and such that $\N_r[\sigma = r]=1$ for $\pi_*$-a.e. $r>0$. (see \cite{p:pmms} for more details for the existence of such a decomposition)

\subsection{The height process and the Lévy CRT}\label{section:height}
We first define a continuum random tree (CRT) using the definition of Aldous \cite{a:crt2,a:crt1,a:crt3}.

\begin{defi}
We say that a metric space $(\ct,d)$ is a real tree if : for $u,v\in \ct$,
\begin{itemize}
\item there exists a unique isometry $\psi_{u,v}:[0,d(u,v)]\rightarrow \ct$ such that $\psi_{u,v}(0)=u$ and $\psi_{u,v}(d(u,v))=v$,
\item if $(w_s,0\leq s\leq 1)$ is an injective path on $\ct$ such that $w_0=u$ and $w_1=v$ then $(w_s,0\leq s\leq1)= \psi_{u,v}([0,d(u,v)])$.
\end{itemize}
A CRT is a random variable $(\ct(\omega),d(\omega))$ on a probability space $(\Omega,\ca,\P)$ such that $(\ct(\omega),d(\omega))$ is a real tree for all $\omega\in \Omega$.
\end{defi}

We can use a height function to define a genealogical structure on a CRT (see Aldous \cite{a:crt3}). Let $g:\R_+ \rightarrow \R_+$ be a function with compact support, non trivial and such that $g(0)=0$. For $s,t\in\ct$, we say that $g(s)$ is the generation of the individual $s$ and that $s$ is an ancestor of $t$ if $g(t)=g_{s,t}$ where
$$g_{s,t} = \inf\{ g(u),s\wedge t\leq u \leq s\vee t \}$$
is the generation of the last common ancestor of the individuals $s$ and $t$.

We define an equivalence relation between two individuals:
$$t\sim t' \hspace{0.5cm} \Longleftrightarrow \hspace{0.5cm} d(t,t'):= g(t)+g(t')-2g_{t,t'}=0.$$
That is to say $g(t)=g_{t,t'}=g(t')$. The quotient set $[0,\sigma]/\sim$ equipped with the distance $d$ and the genealogical relation is then a CRT coded by $g$.\\

Let us now define a height process $H$ associated with the Lévy
process $X$, see Part 1.2 of Duquesne and Le Gall \cite{dlg:rtlpsbp}. For all $t\geq0$, we consider the reversed process at time $t$, $\hat X^{(t)}=(\hat X^{(t)}_s,0\le s\le t)$ defined by :
\[
\hat X^{(t)}_s=
X_t-X_{(t-s)-} \quad \mbox{if}\quad  0\le s<t,
\]
and $\hat X^{(t)}_t=X_t$. We denote by $\hat S^{(t)}$ the supremum process of $\hat X^{(t)}$ and $\hat L^ {(t)}$ the local time at 0 of $\hat S^{(t)} - \hat X^{(t)}$ with the same normalization as in \cite{ad:falp}.

\begin{defi}\label{def:height_process}
  There exists a $[0,\infty  ]$-valued lower semi-continuous process, called the height process such that, under $\N$,
$$
\begin{cases}  
H_0=0,\\
\hbox{for all }t\geq 0\hbox{, a.s. }H_t=\hat L^{(t)}_t. 
\end{cases}$$
And a.s. for all $s<t$ such that $X_{s-}\leq I_t^s$ and for $s=t$, if $\Delta_t>0$ then $H_s<\infty$ and for all $t'>t\geq 0$, the process $H$ takes all the values between $H_t$ and $H_{t'}$ on the time interval $[t,t']$.
\end{defi}

We say that a CRT coded by its associated height process $H$ is a Lévy CRT.

%The height process $(H_t,t\in[0,\sigma])$ under $\N$, codes a continuum genealogical tree, the Lévy CRT (see \cite{adv:plcrt}, Section 2.2).

%-----------version longue---------------
%\begin{itemize}
%\item[(i)] To each $t\in[0,\sigma]$ corresponds a vertex at generation $H_t$.
%\item[(ii)] Vertex $t$ is an ancestor of vertex $t'$ if $H_t=H_{t,t'}$ where $$H_{t,t'}=\inf\{ H_u; u\in[t\wedge t', t\vee t'] \}.$$
%\item[(iii)] We put $d(t,t')=H_t+H_{t'}-2H_{t,t'}$ and we identify $t$ and $t'$ if $d(t,t')=0$.
%\end{itemize}
%The Lévy CRT coded by $H$ is then the quotient set of $[0,\sigma]$ by this equivalence relation, equipped with the distance $d$ and the genealogical relation specified in $(ii)$. 
%----------------------------------------

\subsection{The exploration process}

The height process is not a Markov process in general. But it is a very simple function of a measure-valued Markov process, the exploration process.\\
If $E$ is a locally compact polish space, we denote by $\cb(E)$ (resp. $\cb_+(E)$) the set of $\R$-valued measurable (resp. and non-negative) functions defined on $E$ endowed with its Borel $\sigma$-field, and by $\cm(E)$ (resp. $\cm_f(E)$) the set of $\sigma$-finite (resp. finite) measures on $E$, endowed with the topology of vague (resp. weak)  convergence. For any measure $\mu \in \cm(E)$, and any function $f\in\cb_+(E)$, we write 
$$\left< \mu,f \right> = \int f(x) \mu(dx).$$

The exploration process $\rho=(\rho_t,t\ge    0)$ is a $\cm_f(\R_+)$-valued process defined by, for every $f\in
\cb_+(\R_+) $, $\langle \rho_t,f\rangle =\int_{[0,t]} d_sI_t^sf(H_s)$,
or equivalently
\begin{equation*}\label{eq:def_rho}
\rho_t(dr)=\beta\ind_{[0,H_t]}(r)\; dr + \sum_{\stackrel{0<s\le t}
  {X_{s-}<I_t^s}}(I_t^s-X_{s-})\delta_{H_s}(dr). 
\end{equation*}
In particular, the total mass of $\rho_t$ is 
$\langle \rho_t,1\rangle =X_t-I_t$.

The exploration process also codes the Lévy CRT. Indeed, we can recover the height process $H$ from the exploration process. For $\mu\in \cm(\R_+)$, we put 
\begin{equation*} \label{def:H}
H(\mu)=\sup\, \Supp \mu,
\end{equation*}
where $ \Supp \mu$ is the closed support of $\mu$ with the convention $H(0)=0$. 

To better understand what the exploration process is, let us give some
of its properties. For every $t\geq0$ such
that $\rho_t\neq 0$, the support of the exploration process at time
$t$ is $[0,H_t]$: $\Supp \rho_t=[0,H_t]$. We also have $\rho_t=0$ if
and only if $H_t=0$. We can finally describe the jumps of the
exploration process using the jumps of the Lévy process: $\rho_t= \rho_{t^-} + \Delta_t \delta_{H_t}$, where $\Delta_t=0$ if $t\not\in \cj$. See \cite{dlg:rtlpsbp}, Lemma 1.2.2 and Formula (1.12) for more details.

In the definition of the exploration process, as $X$ starts from 0, we obtain
$\rho_0=0$ a.s.  To state the Markov property of $\rho$, we must first define the process $\rho$ starting at any initial measure $\mu\in
\cm_f(\R_+)$. We recall the notations given in \cite{dlg:rtlpsbp}.

For $a\in  [0, \langle  \mu,1\rangle ] $, we write $k_a\mu$ for the erased measure which is the measure $\mu$ erased by a mass $a$ backward from $H(\mu)$, that is to say:
$$k_a\mu([0,r]) = \mu([0,r])\wedge (\left<  \mu,1  \right>-a), \hbox{ for }r\geq0.$$

In particular, $\left< k_a\mu , 1  \right>=\left< \mu,1 \right>-a$.

For $\nu,\mu \in \cm_f(\R_+)$, and $\mu$ with compact support, we write $[\mu,\nu]\in \cm_f(\R_+) $ for the concatenation of the two measures:
$$\left< [\mu,\nu],f\right> = \left<\mu,f\right> + \left<\nu,f(H(\mu)+\cdot)\right>, \, f\in \cb_+(\R_+).$$

Finally, we put for all $\mu\in  \cm_f(\R_+)$ and for all $t>0$,
\begin{equation*}
   \label{eq:rhot-t-mu}
\rho_t^\mu=\bigl[k_{-I_t}\mu,\rho_t].
\end{equation*}
We say that $(\rho^\mu_t, t\geq
0)$  is the process $\rho$ starting from $\rho_0^\mu=\mu$, and write $\P_\mu$ for its law. Unless there is an ambiguity, we shall write $\rho_t$ for $\rho^\mu_t$. We also denote by $\P_\mu^*$ the law of $\rho^\mu$ killed when it first reaches 0. Then we can state a useful property of the exploration process: the process $(\rho_t,t\ge 0)$ is a càd-làg strong Markov process in $\cm_f(\R_+)$. See \cite{dlg:rtlpsbp}, Proposition 1.2.3 for a proof.

\begin{rem} \label{rm:rho-L} 
As in \cite{ad:falp}, $0$ is also a regular point for $\rho$. Notice that $\N$ is also the excursion measure of the process $\rho$ away from 0, and that $\sigma$, the length of the excursion, is $\N$-a.e. equal to $\inf\{ t>0; \rho_t=0\}$.
\end{rem}

The exponential formula for the Poisson point process of jumps of $\tau$, the inverse subordinator of $-I$, gives (see also the beginning of the Section 3.2.2 \cite{dlg:rtlpsbp}) that for $\lambda>0$
\begin{equation*}\label{eq:N_s}
\N\left[1 -\expp{-\lambda
  \sigma}\right] =\psi^{-1}(\lambda). 
\end{equation*}

\subsection{The dual process and the representation formula}
We shall need the $\cm_f(\R_+)$-valued process $\eta=(\eta_t,t\geq0)$ defined by
\begin{equation*}
\eta_t(dr)=\beta\ind_{[0,H_t]}(r)\; dr + \sum_{\stackrel{0<s\le t}
  {X_{s-}<I_t^s}}(X_{s}-I_t^s)\delta_{H_s}(dr). 
\end{equation*}

This process is called the dual process of $\rho$ under $\N$ (see Corollary 3.1.6 of \cite{dlg:rtlpsbp}). We also denote, for $s\in[0,\sigma]$ fixed, $\kappa_s = \rho_s + \eta_s$. Recall the Poisson representation of $(\rho,\eta)$ under $\N$. Let $\cn(dx \; dl \; du)$ be a point Poisson measure on $[0,+\infty)^3$ with intensity
$$dx \; l\pi(dl) \; \ind_{[0,1]}(u)du.$$
For all $a>0$, we denote by $\M_a$ the law of the pair $(\mu_a,\nu_a)$ of measures on $\R_+$ with finite mass defined by, for any $f\in \cb_+(\R_+)$
$$\left< \mu_a,f\right> = \int \cn(dx \; dl \; du) \ind_{[0,a]}(x) ulf(x) + \beta \int_0^a f(r)dr,$$
$$\left< \nu_a,f\right> = \int \cn(dx \; dl \; du) \ind_{[0,a]}(x) (1-u)lf(x) + \beta \int_0^a f(r)dr.$$
We also put $\M=\int_0^\infty da e^{- \alpha a} \M_a$.

\begin{prop}\label{Representation formula}(\cite{dlg:rtlpsbp}, Proposition 3.1.3)
For every non-negative measurable function $F$ on $\cm_f(\R_+)^2$
$$\N\left[ \int_0^\sigma F(\rho_t, \eta_t)dt\right] = \int \M(d\mu \; d\nu)F(\mu, \nu)$$
where we recall that $\sigma=\inf\{ s>0 ; \rho_s=0 \}$ is the length of the excursion.
\end{prop}

We also give the Bismut formula for the height process of the Lévy
process which gives a spinal decomposition of the tree from a branch
``uniformly randomly'' chosen.
\begin{prop}\label{Bismut}(\cite{dlg:pfalt}, Lemma 3.4.)\\
For every non negative function $F$ defined on $\cb_+([0,\infty])^2$
$$\N\left[ \int_0^\sigma ds F ( ( H_{(s-t)_+},t\geq 0) , ( H_{(s+t)\wedge \sigma},t\geq 0)) \right] = \int \M(d\mu d\nu) \int \P^*_\mu (d\rho) \P^*_\nu (d\tilde{\rho}) F(H(\rho),H(\tilde{\rho}) ).$$
\end{prop}

\section{The Lévy Poison snake}

As in \cite{adv:plcrt}, we construct a Lévy Poisson snake which marks the Lévy CRT on its nodes and on its skeleton. The aim is to fragment the CRT in several fragments using point processes whose intensities depend on a parameter $\theta$ such that, if $\theta=0$, there is no marks on the CRT and the number of marks increases with $\theta$.

%Thus we get the Laplace exponent of a subordinator by :
%
%\begin{eqnarray*}
% \psi^{(\theta)}(\lambda) - \psi(\lambda) & = &  2\beta\theta \lambda + \int_{(0,+\infty)}(1-e^{-\lambda l})(1-e^{-\theta l})\pi(dl).
%\end{eqnarray*}
%
%This subordinator gives us the intensity to construct the two processes $m^{ske}$ and $m^{nod}$ which put marks on the skeleton and on the nodes of the CRT. We keep the coefficient $2\beta$ to construct a Lévy Poisson snake and put marks along the skeleton of the CRT. To mark the nodes of the CRT, we follow the construction made in \cite{ad:falp}.

\subsection{Marks on the skeleton}\label{marquesquelette}

In order to mark the continuous part of the CRT and to keep track of marks along the lineage of each individual, we construct a snake on $E=\mathcal M(\mathbb R^2_+)$ where the parameter $\theta$ appears. To obtain a Polish space, we separate the space of the parameter $\theta$ in bounded intervals.\\
We fix $i\in \N$, thanks to \cite{d:mvmp} Section 3.1, $E_i=\mathcal M_f(\R_+ \times [i,i+1))$ the set of finite measures on $\R_+\times [i,i+1)$ is a Polish space for the topology of weak convergence.\\
Thanks to \cite{dlg:rtlpsbp}, Chap. 4, there exists a $E_i$-valued process $(W_t^i, t\geq0)$ such that conditionally on $X$,
 \begin{enumerate}
\item  For each $s\in[0,\sigma]$, $\displaystyle W_s^i$ is a Poisson measure on $[0,H_s]\times[i,i+1)$ with intensity $2\beta\ind_{[0,H_t]}(r)dr~\ind_{[i,i+1)}(\theta)d\theta$,
\item  For every $s<s'$, $W^i_{s'}(dr,d\theta)\ind_{[0,H_{s,s'}]}(r)=W^i_{s}(dr,d\theta)\ind_{[0,H_{s,s'}]}(r)$,
\end{enumerate}
where we recall that $H_{s,s'}=\inf_{[s,s']}H$.

We take the processes $W^i$ independently and we set $m_t^{ske} = \sum_{i\in\N}W^i_t$.\\
If $\beta=0$, the CRT has no Brownian part, in this case, there is no mark on the skeleton and we set $m^{ske}=0$.\\
For $t\geq0$ fixed, conditionally on $H_t$, $m_t^{ske}$ is Poisson point process with intensity 
$$2\beta \ind_{[0,H_t]}(r)dr d\theta.$$
The process $(\rho,m^{ske})$ takes values in the space $\tilde{\cm}_f := \cm_f(\R_+)\times  \cm(\R^2_+)$. We denote by $(\cf_s,s\geq0)$ the canonical filtration on the space of càd-làg trajectories on the space $\tilde{\cm}_f$.\\
Using Theorem 4.1.2 of \cite{dlg:rtlpsbp} when $H$ is continuous or the adapted result  when $H$ is not continuous (Prop. 7.2, \cite{adv:plcrt}), we get the following result
\begin{prop}\label{markovske}
$(\rho,m^{ske})$ is a strong Markov process with respect to the filtration $(\cf_{s+},s\geq 0)$.
\end{prop}

\subsection{Mark on the nodes}\label{marquesnoeuds}
We mark every jump of  the process $X$, say $s$ such that $\Delta_s>0$, with an independent Poisson measure with intensity  $\Delta_s \ind_{u>0}du$, and this point Poisson measure is denoted by $\displaystyle  \sum_{u>0}\delta_{V_{s,u}}  $.

When the Lévy measure of $X$ is non trivial, we define the mark process on the nodes of the CRT as in \cite{ad:falp}. We use a Poisson point measure to introduce the parameter $\theta$. Conditionally on $X$, we set

\[
m^{\text{nod}}_t    (dr,d\theta)   =   \sum_{\stackrel{0<s\le    t}
  {X_{s-}<I_t^s}}\left(I_t^s-X_{s-}\right)    \left(   \sum_{u>0}\delta_{V_{s,u}} (d\theta)   \right)     \delta_{H_s}(dr).
\]
If $\pi=0$, it is the Brownian case and there is no mark on the nodes, thus we set $m^{nod}=0$.

\subsection{The snake}
We join the marks on the skeleton and the marks on the nodes of the CRT in a mark process $m=(m^{nod},m^{ske})$. We write $\cs=(\rho,m)$ the marked snake starting from $\rho_0=0$ and $m_0=0$.\\
Let us recall the construction made in \cite{adv:plcrt} to obtain a snake starting from an initial value and then to write a strong Markov property for the snake. We consider the set $\S$ of triplets $(\mu,\Pi^{nod},\Pi^{ske})$ such that
\begin{itemize}
\item $\mu\in \cm_f(\R_+)$,
\item $\Pi^{nod}$ can be written as $\Pi^{nod}(dr,dx) = \mu(dr)\Pi^{nod}_r(dx)$ where $(\Pi^{nod}_r,r>0)$ is a family of $\sigma$-finite measures on $\R_+$ and for every  $\theta>0$, $\Pi^{nod}(\R_+ \times [0,\theta])<\infty$,
\item $\Pi^{ske} \in \cm(\R_+^2)$ and 
\begin{itemize}
\item $Supp(\Pi^{ske}(.,\R_+)) \subset Supp(\mu)$
\item for every $x<H(\mu)$ and every $\theta>0$, $\Pi^{ske}([0,x]\times [0,\theta])<\infty$,
\item if $\mu(H(\mu))>0$, then for every $\theta>0$, $\Pi^{ske}(\R_+\times [0,\theta])<\infty$
\end{itemize}
\end{itemize}

Then we define the snake $\cs$ starting from an initial value $(\mu,\Pi)\in\S$, where $\Pi = (\Pi^{nod},\Pi^{ske})$. That is to say
$$\cs_0^{(\mu,\Pi)} := (\rho^{\mu}_0,(m^{nod})_0^{(\mu,\Pi)},(m^{ske})_0^{(\mu,\Pi)} ) = (\mu,\Pi).$$

 We write $H^\mu_t = H(k_{-I_t}\mu)$ and $H_{0,t}^\mu=\inf\{ H_u^\mu ; u\in [0,t] \}$. We define
$$(m^{nod})_t^{(\mu,\Pi)} = \left[ \Pi^{nod}\ind_{[0,H_t^\mu)} + \ind_{\mu(\{ H^\mu_t \})>0} \frac{k_{-I_t}\mu(\{ H_t^\mu \}) \Pi^{nod}(\{H_t^\mu\},.) }{ \mu(\{ H_t^\mu \}) }\delta_{H^\mu_t}\Pi^{nod}_{H^\mu_t},m_t^{nod} \right]$$
$$and \hspace{1cm}(m^{ske})_t^{(\mu,\Pi)} = \left[ \Pi^{ske}\ind_{[0,H^\mu_{0,t})} , m_t^{ske} \right].$$
Notice that these definitions are coherent with the previous definitions of the processes $m^{nod}$ and $m^{ske}$.\\
By using the strong Markov property for the process $(\rho,m^{nod})$ (see \cite{ad:falp}, Proposition 3.1) and Proposition \ref{markovske}, we obtain that the snake $\cs$ is a càd-làg strong Markov process. See Proposition 2.5 of \cite{adv:plcrt}.

%By contruction, we get the snake property : a.s.
%$$(\rho_t,m_t)(ds)\ind_{s<H_{t,t'}}=(\rho_{t'},m_{t'})(ds)\ind_{s<H_{t,t'}}.$$
We write $m^{(\theta)}(dr)=m^{ske}(dr,[0,\theta])+m^{nod}(dr,[0,\theta])$. Due to the properties of the Poisson point measures, we obtain the following result.

\begin{prop}
$m_t^{(\theta+\theta')}-m_t^{(\theta)}$ is independent of $m_t^{(\theta)}$ and has the same law as $m_t^{(\theta')}$.
\end{prop}

We still denote by $\P_\mu$ (resp. $\P^*_\mu$) the law of the snake $(\rho,m^{nod},m^{ske})$ starting from $(\mu,0,0)$ (resp. and killed when it reaches 0). We also denote by $\N$ the law of the snake $\cs$ when $\rho$ is distributed under $\N$.

We define $\psi^{(\theta)}$ by, for any $\theta \in \R$,
\begin{eqnarray*}
\psi^{(\theta)}(\lambda) &=& \psi(\theta+\lambda)-\psi(\theta)\\
&=& \alpha^{(\theta)} \lambda + \beta^{(\theta)} \lambda^2 + \int_{(0,+\infty)}(e^{-\lambda l}-1+\lambda l) \pi^{(\theta)}(dl)
\end{eqnarray*}
$$with\begin{cases}
\alpha^{(\theta)} =  \alpha+2\beta\theta + \int_{(0,+\infty)}(1-e^{-\theta l})l\pi(dl)\\
\beta^{(\theta)}=\beta\\
\pi^{(\theta)}(dl)=e^{-\theta l}\pi(dl).
\end{cases}$$

For fixed $\theta\geq0$ and $t\in [0,\sigma]$, we define the set $A^{(\theta)}_t$ of individuals of the Lévy CRT without marks on their lineage and its right-continuous inverse $C^{(\theta)}_t$ given by the formulas:
$$A^{(\theta)}_t=\int_0^t \ind_{m^{(\theta)}_s=0}ds \hspace{0.5cm}and \hspace{0.5cm} C_t^{(\theta)} = \inf\{  s>0 ; A_s^{(\theta)}>t  \}.$$
We define the exploration process $\rho^{(\theta)}$ which describes the sub tree under the first marks given by $m^{(\theta)}$ : $\rho^{(\theta)}_t=\rho_{C^{(\theta)}_t}$. Let $\cf^{(\theta)}=(\cf^{(\theta)}_t,t\geq0)$ be the filtration generated by pruned Lévy Poisson snake $\cs^{(\theta)} = (\rho^{(\theta)},m^{(\theta)})$ completed the usual way. We also denote $\sigma{(\theta)}=\inf \{ t>0; \rho_t^{(\theta)}=0\}$ and $X^{(\theta)}$ the Lévy process with Laplace exponent $\psi^{(\theta)}$.

We can write the key property of $\rho^{(\theta)}$ proved by Abraham, Delmas and Voisin \cite{adv:plcrt}.

\begin{prop}[Theorem 1.1 \cite{adv:plcrt}]\label{explorationtheta}
The exploration process $\rho^{(\theta)}$ is associated with a Lévy process with Laplace exponent $\psi^{(\theta)}$.
\end{prop}

The next Lemma is also crucial for getting a fragmentation process and
explains the choice of the parameters of the pruning. It has been proved by Abraham and Delmas \cite{ad:falp}, see the comments under their Lemma 1.6. Notice that the proof of Abraham and Delmas is established in the general case when the quadratic coefficient $\beta$ is nonnegative.

\begin{lem}\label{lem:excursions}
For $\pi_*(dr)$ a.e. $r$, the ``law" of $\rho^{(\theta)}$ under $\N$, conditionally on
$\sigma^{(\theta)}=r$ is the same as the "law" of $\rho$ under $\N$,
conditionally on $\sigma=r$. 
\end{lem}

\subsection{Poisson representation of the snake}\label{represpoisson}
We decompose the process $\rho$ under $\P_{\mu}^*$ according to excursions of the total mass of $\rho$ above its past minimum. More precisely, let $(\alpha_i,\beta_i),i\in J$ be the excursion intervals of $X-I$ above 0 under $\P^*_{\mu}$. For $i\in J$, we define $h_i=H_{\alpha_i}$ and $\rho^i$ by the formula : for $t\geq0$ and $f\in \cb_+(\R_+)$,
$$\left< \rho_t^i,f  \right> = \int_{(h_i,+\infty)} f(x-h_i) \rho_{(\alpha_i + t) \wedge \beta_i}(dx).$$
We write $\sigma^i = \inf \{ s>0; \left< \rho^i_s ,1 \right> =0 \}$.\\
We also define the mark process $m$ above the intervals $(\alpha_i,\beta_i)$. For every $t\geq 0$ and $f\in \cb_+(\R_+^2)$, we set 
$$\left< m_t^{i,a} , f \right> = \int_{(h_i,+\infty)} f(x-h_i,\theta) m^a_{(\alpha_i + t) \wedge \beta_i}(dx,\theta)$$
with $a=ske,nod$. We set for all $i\in J$,
$m^i=(m^{i,nod},m^{i,ske})$. It is easy to adapt the proof of Lemma
4.2.4 of \cite{dlg:rtlpsbp} to get the following Poisson representation.

\begin{lem}\label{decomp poisson}
Let $\mu\in \cm_f(\R_+) $. The point measure $\displaystyle \sum_{i\in J}\delta_{(h_i,\cs^i)}$ is under $\P^*_{\mu}$ a Poisson point measure with intensity $\mu(dr)\N(d\cs)$.
\end{lem}

\subsection{Special Markov property}

We fix $\theta\geq0$. We define $O^{(\theta)}$ as the interior of the set
$$\{ s\ge 0,\ m_s^{(\theta)}\ne 0\}.$$
We write $O^{(\theta)}=\bigcup_{i \in \tilde I}(a_i,b_i)$ and we say
that $(a_i,b_i)$ are the excursions intervals of the Lévy marked snake
$\cs^{(\theta)}=(\rho^{(\theta)},m^{(\theta)})$ away from $\{s\geq0 ;
m_s^{(\theta)}=0  \}$. We set $h_i=H_{a_i}$ and we define the process
$\cs^{(\theta),i}=(\rho^{(\theta),i},m^{(\theta),i})$ above the
excursion intervals $((a_i,b_i),i\in \tilde I)$ as previously.

If $Q$ is a measure on $\S$ and $\varphi$ is a non-negative measurable function defined on the measurable space $\R_+\times\Omega\times \S$, we denote by
$$Q[\varphi(u,\omega,\cdot)] = \int_{\S}\varphi(u,\omega,\cs)Q(d\cs).$$

We now recall the special Markov property proved by Abraham, Delmas and Voisin \cite{adv:plcrt}. It gives the distribution of the Lévy snake "above" the "first" marks of the marked CRT knowing the part of the pruned CRT where the root belongs to.

\begin{theo}[\cite{adv:plcrt}, Theorem 4.2 ](Special Markov property)\label{special}\\
We fix $\theta>0$. Let $\phi$ be a non-negative measurable function defined on $\R_+\times \S$ such that $t\mapsto \phi(t,\omega,\cs)$ is progressively $\cf^{(\theta)}_\infty$-measurable for any $\cs\in \S$. Then, we have $\P$-a.s.

\begin{multline}
   \label{eq:MS}
\E\left[\exp\left(-\sum_{i\in\tilde I}\varphi(A^{(\theta)}_{a_i},\omega,\cs^{(\theta),i})\right)\biggm| \cf^{(\theta)}_\infty \right]\\ 
=\exp\left(-\int_0^\infty     du\,2\beta\theta \N\left[1-\expp{-\varphi(u,\omega,\cdot)}\right]   \right) \\
\exp\left(-\int_0^\infty     du\,\int_{(0,\infty )}(1-e^{-\theta\ell}) \pi(d\ell) \left(1-   \E^*_\ell[\expp{-\varphi(u,\omega,\cdot)} ] \right) \right).  
\end{multline}

Furthermore, the law of the excursion process $\displaystyle
\sum_{i\in \tilde I}\delta_{ ( A^{(\theta)}_{a_i},\rho^{(\theta)}_{a_i-} ,\cs^{(\theta),i}) }$, given $\cf^{(\theta)}_\infty$, is the law of a Poisson point measure with intensity $\displaystyle \ind_{u\geq0} du \, \delta_{\rho_u^{(\theta)}}(d\mu) \left( 2\beta\theta \N(d\cs)+\int_{(0,\infty)} (1-e^{-\theta\ell}) \pi(d\ell)\P_\ell^*(d\cs) \right)$.

\end{theo}

\section{Links between the snake and the fragmentation}
\subsection{Construction of the fragmentation process}\label{relation}

%------------------------version courte------------------------
We are interested in the fragments of the tree given by the marks process. We do the same construction as in \cite{ad:falp}, Section 4.1.

For fixed $\theta\geq0$, we first construct an equivalence relation,$\mathcal R_\theta$, on $[0,\sigma]$ under $\N$ or under $\N_\sigma$ by :
$$ s \mathcal R_\theta t \Leftrightarrow m_s^{(\theta)}([H_{s,t},H_s])=m_t^{(\theta)}([H_{s,t},H_t])=0 .$$
Two individuals, $s$ and $t$, belong to the same equivalence class if they belong to the same fragment, that is to say if there is no mark on their lineage down to their most recent  common ancestor. From the equivalence relation $\mathcal R_\theta$, we get the family of sets $G^j$ of individuals with $j$ marks in their lineage. 

As we put marks on infinite nodes of the CRT, for $\theta>0$, for fixed $j\in \N$, the set $G^j$ can be written as an infinite union of sub-intervals of $[0,\sigma]$. We get
$$G^j=\bigcup_{k\in J_j}R^{j,k}$$
such that $R^{j,k}$ has positive Lebesgue measure. For $j\in \N$ and $k\in J_j$, we set
$$A_t^{j,k}=\int_0^t \ind_{s\in R^{j,k}} ds \hspace{0.3cm} \hbox{and} \hspace{0.3cm} C^{k,j}_t=\inf\{   u\geq 0; A_u^{j,k}>t  \} ,$$
with the convention $\inf \emptyset = \sigma$. We also construct the process $\tilde{\cs}^{j,k}=(\tilde{\rho}^{j,k},\tilde{m}^{j,k})$ by : for every $f\in \cb_+(\R_+)$,
 $\varphi\in \cb_+(\R_+,\R_+)$ and $t\geq0$,
$$\left<  \tilde{\rho}_t^{j,k},f   \right> = \int_{(H_{C_0^{j,k}},+\infty )} f(x-H_{C_0^{j,k}}) \rho_{C_t^{j,k}}(dx)$$
$$\left<  \tilde{m}_t^{j,k},\varphi   \right>  = \int_{(H_{C_0^{j,k}},+\infty )\times (\theta,+\infty)} \varphi(x-H_{C_0^{j,k}},v-\theta) m_{C_t^{j,k}}(dx,dv)$$
$\tilde{\sigma}^{j,k}$ corresponds to the Lebesgue measure of $R^{j,k}$.

We denote $\cl^{(\theta)}=(\tilde \rho^{j,k}; j\in \N, k\in J_j)=( \rho^i; i\in I^{(\theta)})$. We also define $\cl^{(\theta-)}=(\rho^i; i\in I^{(\theta -)})$ the set defined similarly but using the equivalence relation $\mathcal R_{\theta-}$ which gives the fragments just before time $\theta$.

We now define the process $\Lambda^\theta = (\Lambda^\theta_1,\Lambda^\theta_2,\dots )$ as the sequence of non trivial Lebesgue measure of the equivalence classes of $\mathcal R_\theta$, $(\tilde{\sigma}^{j,k}, j\in \N, k\in J_j)$, ranked in decreasing order. Notice that, when $\theta>0$, this sequence is infinite. When $\theta=0$, $\Lambda^0$ is the entire tree and we denote $\Lambda^0=(\Lambda^0,0,\dots)$. Then we have that $\N$-a.s. and $\N_\sigma$-a.e.
$$\Lambda^\theta \in \cs^{\downarrow}.$$
We write $P_{\sigma}$ the law of $(\Lambda^\theta,\theta\geq0)$ under $\N_\sigma$ and by convention $P_\mathbf{0}$ is the Dirac mass at $(0,0,\dots)\in \cs^{\downarrow}$.

\begin{theo}\label{theofragmentation}
For $\pi_*(dr)$-almost every $r$, under $P_r$, $(\Lambda^\theta,\theta\geq 0)$ is a $\cs^{\downarrow}$-valued fragmentation process.
\end{theo}

\begin{proof}[Sketch of proof] As the proof is exactly the same
  as for \cite{ad:falp}, Theorem 1.1., we only give the main ideas of
  the proof and refer to \cite{ad:falp} for precise details.

{\it 1st step.} Thanks to Proposition \ref{explorationtheta}, the
first sub-excursion $\tilde \rho^{0,0}$ is ``distributed'' under $\N$
as $\rho^{(\theta)}$. Moreover, by Lemma \ref{lem:excursions} and the
construction of the Poisson snake conditionally given $\rho$, the law
of $\tilde \cs^{0,0}$ conditionally on $\{\tilde\sigma^{0,0}=s\}$ is
$\N_{s}$.

{\it 2nd step.} By the special Markov property, Theorem \ref{special},
using the same notations as in this theorem, we have that,
under $\N$ conditionally on $\tilde \sigma^{0,0}=s$, the processes
$(\cs^{(\theta),i},i\in\tilde I)$ are given by a Poisson measure with
intensity
$$s2\beta\theta\N(d\cs)+s\int_{(0,+\infty)}(1-e^{-\theta\ell})\pi(d\ell)\P_\ell^*(d\cs)$$
and are independent of $\tilde\cs^{0,0}$.

Moreover, by the Poisson representation of the probability measure
$\P_\ell^*$ (Lemma \ref{decomp poisson}), we get that the 
excursions of the snake $\cs$ ``above the first mark'', that we denote
$(\cs^{1,k},k\in J_1)$, form under $\N$ an i.i.d. family of processes
with ``distribution'' $\N$.

{\it 3rd step} By induction on the number of marks, we get the
following lemma

\begin{lem}\label{sjk}
Under $\N$, the law of the family $(\tilde{\cs}^{j,k}, j\in \N, k\in
J_j)$, conditionally on $(\tilde{\sigma}^{j,k}, j\in \N, k\in J_j)$, is
the law of independent Lévy Poisson snakes distributed respectively as $\N_{\tilde{\sigma}^{j,k}}$.
\end{lem}

The theorem then follows easily.
\end{proof}

\subsection{Another representation of the fragmentation}\label{another}
We give an another representation of the fragmentation by using a Poisson point measure under the epigraph of the height process. Recall that for every $t \in[0,\sigma]$, $$\kappa_t(dr) =2\beta\ind_{[0,H_t]}(r)(dr)+\sum_{\stackrel{0<s\le t}{X_{s-}<I_t^s}}(X_{s}-X_{s-})\delta_{H_s}(dr).$$

%We cut the marked nodes of the CRT, that is to say the set of points $(s,a)$ such that $\kappa_s(\{a\})>0$ and we cut on the skeleton of the CRT, that is to say the set of points $(s,a)$ with $0\leq a \leq H_s$.

Conditionally on the process $H$ (or equivalently on $\rho$), we set a Poisson point process $\cq(d\theta,ds,da)$ under the epigraph of $H$ with intensity $d\theta \; q_\rho(ds,da)$ where 
\begin{eqnarray*}
q_\rho(ds,da)&=& \frac{ds~\kappa_s(da)}{d_{s,a}-g_{s,a}}\\
&=&q_\rho^{ske}(ds,da)+q_\rho^{nod}(ds,da)\\
\end{eqnarray*}
$$\hbox{ with }\begin{cases}
q_\rho^{nod}(ds,da) = \displaystyle\frac{ds}{d_{s,a}-g_{s,a}}\sum_{\stackrel{0<u\le s}  {X_{u-}<I_s^u}}(X_{u}-X_{u-})\delta_{H_u}(da)\\
q_\rho^{ske}(ds,da) = \displaystyle\frac{2\beta \; ds \; \ind_{[0,H_s]}(a)da}{d_{s,a}-g_{s,a}}
\end{cases}$$
with $d_{s,a}=\sup\{ u\geq s , \min\{ H_v, v\in[s,u]\} \geq a  \}$ and $g_{s,a}=\inf\{ u\leq s , \min\{ H_v, v\in[s,u]\}\geq a  \}$. $[g_{s,a},d_{s,a}]$ is the set of individuals of the CRT with common ancestor $s$ after generation $a$.\\

\begin{prop}
Conditionally on the process $H$, the mark process $m$ and the Poisson point process $\cq$ have same distribution

\end{prop}

\begin{proof}
Conditionally on $H$ and under $\N$, for fixed $t\in [0,\sigma]$, $m_t^{\text{ske}}$ is a Poisson point process with intensity $2\beta \ind_{[0,H_t]}(a)da\, d\theta$. For an individual $t$, marks on the skeleton are uniform from the height $h_1$ to the height $h_2$. Thanks to snake property, if a mark appears at height $a$ on the lineage of the individual $t$, it also appears for all children of $t$, that is to say the mark appears from $g_{t,a}$ to $d_{t,a}$. We have 

$$\iint_D \frac{2\beta \theta da \, d s}{d_{s,a}-g_{s,a}}= 2\beta(h_2-h_1)$$

where $D=\{(a,s)\in[h_1,h_2]\times [g_{t,a},d_{t,a}]\}$ because for $a\in[h_1,h_2]$ fixed, $\forall s\in[g_{t,h},d_{t,h}]$, $g_{s,a}=g_{t,a}$ and $d_{s,a}=d_{t,a}$. (See Figure \ref{fig:bandemarque}). Thus the point process $q^{\text{ske}}$ give the same marks as the process $m^{\text{ske}}$.

\begin{center}
\begin{figure}[h!]
\caption{Marks under the epigraph of $H$}\label{fig:bandemarque}

\scalebox{0.8} % Change this value to rescale the drawing.
{
\begin{pspicture}(0,-3.0367188)(15.019188,3.0137188)
\psline[linewidth=0.034cm,arrowsize=0.05291667cm 2.0,arrowlength=1.4,arrowinset=0.4]{->}(0.5221875,-2.4832811)(15.002188,-2.6032813)
\psline[linewidth=0.034cm,arrowsize=0.05291667cm 2.0,arrowlength=1.4,arrowinset=0.4]{->}(0.7421875,-2.8432813)(0.7421875,2.9967186)
\psline[linewidth=0.034](0.7621875,-2.4432812)(0.9221875,-1.8232813)(1.0421875,-2.1632812)(1.2821875,-1.3432813)(1.4421875,-1.8432813)(1.4821875,-1.5832813)(1.6421875,-2.0232813)(1.9621875,-0.80328125)(2.1221876,-1.2032813)(2.3821876,0.37671876)(2.5221875,-0.46328124)(2.6221876,-0.08328125)(2.7421875,-0.48328125)(2.9621875,0.61671877)(3.1621876,-0.26328126)(3.3221874,0.25671875)(3.5221875,-0.48328125)(3.8221874,0.7367188)(4.1221876,-0.20328125)(4.2421875,0.03671875)
\psline[linewidth=0.034](4.2221875,0.03671875)(4.3821874,-0.46328124)(4.5421877,0.19671875)(4.7421875,-0.48328125)(4.9021873,0.05671875)(5.0021877,-0.16328125)(5.3221874,0.83671874)(5.5221877,0.39671874)(5.6021876,0.61671877)(5.8021874,0.27671874)(6.1621876,1.6167188)(6.5021877,0.75671875)(6.6021876,1.0767188)(7.0221877,0.15671875)(7.1421876,0.43671876)(7.3821874,-0.12328125)(7.5621877,0.63671875)(7.8821874,-0.32328126)(7.9821873,0.05671875)(8.222187,-0.50328124)(8.422188,0.17671876)(8.642187,-0.5232813)(8.822187,-0.08328125)(8.942187,-0.32328126)(9.362187,0.87671876)(9.542188,0.5367187)(9.722187,1.1767187)(10.122188,0.01671875)(10.2421875,0.35671875)(10.582188,-0.5232813)(10.762188,-0.02328125)(11.202188,-1.2232813)(11.282187,-0.86328125)(11.502188,-1.4032812)(11.7421875,-0.16328125)(11.842188,-0.40328124)(12.002188,0.13671875)(12.382188,-1.5632813)(12.542188,-1.1832813)(12.7421875,-1.7432812)(12.942187,-1.0432812)(13.202188,-2.0232813)(13.322187,-1.6232812)(13.662188,-2.5632813)(13.682187,-2.5432813)
\usefont{T1}{ptm}{m}{n}
\rput(13.869687,-2.8232813){\small $\sigma$}
\psline[linewidth=0.024cm](2.1821876,-0.9032813)(11.082188,-0.9032813)
\psline[linewidth=0.024cm](5.0221877,-0.08328125)(7.3621874,-0.06328125)
\psline[linewidth=0.024cm](4.9821873,-0.14328125)(4.8621874,-0.14328125)
\psline[linewidth=0.024cm](4.7221875,-0.50328124)(2.2221875,-0.50328124)
\psline[linewidth=0.024cm](7.3821874,-0.14328125)(7.8221874,-0.14328125)
\psline[linewidth=0.024cm](7.8621874,-0.34328124)(8.142187,-0.34328124)
\psline[linewidth=0.024cm](8.222187,-0.50328124)(10.942187,-0.50328124)
\psline[linewidth=0.02cm](2.1821876,-0.74328125)(2.3621874,-0.50328124)
\psline[linewidth=0.02cm](2.3021874,-0.9032813)(2.5821874,-0.50328124)
\psline[linewidth=0.02cm](2.5821874,-0.9032813)(2.8621874,-0.50328124)
\psline[linewidth=0.02cm](2.8421874,-0.9032813)(3.1221876,-0.50328124)
\psline[linewidth=0.02cm](3.0621874,-0.92328125)(3.3421874,-0.5232813)
\psline[linewidth=0.02cm](3.3421874,-0.92328125)(3.6221876,-0.5232813)
\psline[linewidth=0.02cm](3.6021874,-0.92328125)(3.8821876,-0.5232813)
\psline[linewidth=0.02cm](3.8821876,-0.9032813)(4.1621876,-0.50328124)
\psline[linewidth=0.02cm](4.1621876,-0.9032813)(4.4421873,-0.50328124)
\psline[linewidth=0.02cm](4.4221873,-0.9032813)(4.6821876,-0.50328124)
\psline[linewidth=0.02cm](4.6621876,-0.9032813)(5.1421876,-0.08328125)
\psline[linewidth=0.02cm](4.9221873,-0.9032813)(5.4421873,-0.06328125)
\psline[linewidth=0.02cm](5.2021875,-0.9032813)(5.7021875,-0.08328125)
\psline[linewidth=0.02cm](5.4821873,-0.9032813)(5.9821873,-0.08328125)
\psline[linewidth=0.02cm](5.7621875,-0.9032813)(6.2621875,-0.06328125)
\psline[linewidth=0.02cm](6.0221877,-0.9032813)(6.5421877,-0.06328125)
\psline[linewidth=0.02cm](6.3021874,-0.9032813)(6.8021874,-0.06328125)
\psline[linewidth=0.02cm](6.5821877,-0.9032813)(7.0821877,-0.06328125)
\psline[linewidth=0.02cm](6.8421874,-0.9032813)(7.3621874,-0.06328125)
\psline[linewidth=0.02cm](7.1221876,-0.9032813)(7.5821877,-0.14328125)
\psline[linewidth=0.02cm](7.3821874,-0.9032813)(7.8221874,-0.18328124)
\psline[linewidth=0.02cm](7.6621876,-0.9032813)(8.022187,-0.34328124)
\psline[linewidth=0.02cm](7.9221873,-0.9032813)(8.202188,-0.46328124)
\psline[linewidth=0.02cm](8.202188,-0.9032813)(8.482187,-0.50328124)
\psline[linewidth=0.02cm](8.462188,-0.9032813)(8.7421875,-0.50328124)
\psline[linewidth=0.02cm](8.702188,-0.9032813)(8.982187,-0.50328124)
\psline[linewidth=0.02cm](8.982187,-0.9032813)(9.2421875,-0.50328124)
\psline[linewidth=0.02cm](9.2421875,-0.9032813)(9.502188,-0.50328124)
\psline[linewidth=0.02cm](9.502188,-0.9032813)(9.782187,-0.50328124)
\psline[linewidth=0.02cm](9.762188,-0.9032813)(10.042188,-0.50328124)
\psline[linewidth=0.02cm](10.042188,-0.9032813)(10.302188,-0.50328124)
\psline[linewidth=0.02cm](10.282187,-0.9032813)(10.571981,-0.51417357)
\psline[linewidth=0.02cm](10.571981,-0.9032813)(10.861776,-0.51417357)
\psline[linewidth=0.02cm](10.841077,-0.9032813)(11.002188,-0.66328126)
\psline[linewidth=0.03cm](6.4421873,1.5567187)(6.4421873,-2.7432814)
\psline[linewidth=0.024cm](8.182187,-0.42328125)(4.7821875,-0.42328125)
\psline[linewidth=0.024cm](0.8421875,-0.10328125)(0.6221875,-0.10328125)
\psline[linewidth=0.024cm](0.8221875,-0.88328123)(0.6421875,-0.88328123)
\psline[linewidth=0.024cm](0.8221875,-0.42328125)(0.6221875,-0.42328125)
\usefont{T1}{ptm}{m}{n}
\rput(0.4196875,0.05671875){\small $h_2$}
\usefont{T1}{ptm}{m}{n}
\rput(0.3996875,-0.7032812){\small $h_1$}
\usefont{T1}{ptm}{m}{n}
\rput(0.3296875,-0.34328124){\small $a$}
\usefont{T1}{ptm}{m}{n}
\rput(6.6296873,-2.7632813){\small $t$}
\psline[linewidth=0.024cm,linestyle=dashed,dash=0.16cm 0.16cm](4.7821875,-0.40328124)(4.7821875,-2.7032812)
\psline[linewidth=0.024cm,linestyle=dashed,dash=0.16cm 0.16cm](8.162188,-0.42328125)(8.182187,-2.7032812)
\usefont{T1}{ptm}{m}{n}
\rput(5.2096877,-2.7832813){\small $g_{t,a}$}
\usefont{T1}{ptm}{m}{n}
\rput(8.369687,-2.8232813){\small $d_{t,a}$}
\end{pspicture} 
}

\end{figure}
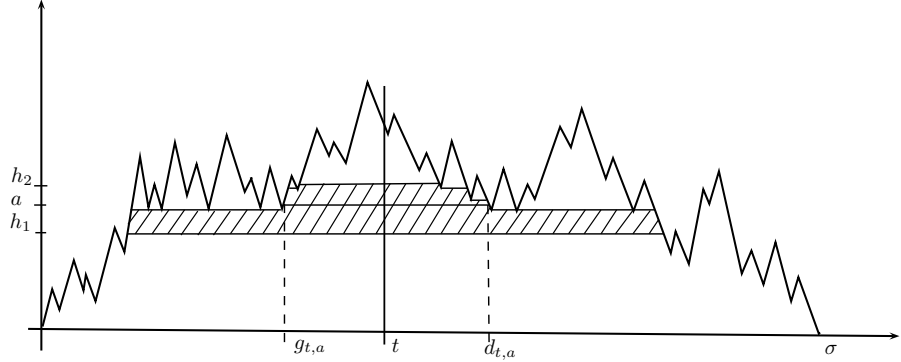
\end{center}

Conditionally on $H$, under $\N$, for fixed $t\in [0,\sigma]$, $m_t^{\text{nod}}$ puts marks on a node at height $H_u<t$ proportionally to the size $\Delta_u$ of the node. By construction of $m^{\text{nod}}$, if a mark appears at height $H_u=a$ on the lineage of an individual $t$, it also appears from $g_{t,a}$ to $d_{t,a}$. We have

$$\iint_D \frac{\Delta_u \theta \delta_{H_u}( da) ds}{d_{s,a}-g_{s,a}} = \theta\Delta_{u} $$

where $D=\{(a,s)\in[0,H_t]\times [g_{t,a},d_{t,a}]\}$. Thus the point process $q^{\text{nod}}$ give the same marks as the process $m^{\text{nod}}$.

The equality $m=m^{\text{ske}}+m^{\text{nod}}$ ends the proof.
\end{proof}

We use a notation  for the fragments of the CRT obtained from a mark $(s,a)$ under the epigraph of $H$. For $s$ and $a$ such that $s\in[0,\sigma]$ and $0\leq a\leq H_s$, we denote the fragments of the Lévy snake $(\rho^i,i\in \tilde I)$ by :
\begin{itemize}
\item the open intervals of the excursion of $H$ after $s$ and above $a$ : $((\alpha_i,\beta_i),i\in \tilde I_+)$ which are such that $\alpha_i>s$, $H_{\alpha_i}=H_{\beta_i}=a$ and for every $s'\in(\alpha_i,\beta_i)$, $H_{s'}>a$ and $H_{s,s'}=a$.
\item the open intervals of the excursion of $H$ before $s$ and above $a$ : $((\alpha_i,\beta_i),i\in \tilde I_-)$ which are such that $\beta_i<s$, $H_{\alpha_i}=H_{\beta_i}=a$ and for every $s'\in(\alpha_i,\beta_i)$, $H_{s'}>a$ and $H_{s,s'}=a$.
\item the excursion $i_s$, of $H$ above $a$ and which contains $s$ : $(\alpha_{i_s},\beta_{i_s})$ such that $\alpha_{i_s}<s<\beta_{i_s}$, $H_{\alpha_{i_s}}=H_{\beta_{i_s}}=a$ and for every $s'\in(\alpha_{i_s},\beta_{i_s})$, $H_{s'}>a$ and $H_{s,s'}>a$.
\item the excursion $i_0$ of $H$ without the mark  $(s,a)$ : $\{ s\in[0,\sigma]; H_{s,s'}<a \}=[0,\alpha_{i_0})\cup (\beta_{i_0},\sigma]$
\end{itemize}
We write $\tilde I=\tilde I_- \cup \tilde I_+ \cup \{ i_s,i_0 \}$ (see Figure \ref{figure:fragments}). Then the family $(\rho^i,i\in \tilde{I})$ contains the exploration processes of the fragments obtained when a single cutpoint is selected. 
%For the mark on the skeleton $(s,a)$, the set $\tilde I_-\cup \tilde I_+$ is empty.\\
We are interested in the computation of $\tilde \nu_\rho$, the "law" of $(\rho^i,i\in \tilde I)$ under $\N(d\rho)q_\rho(ds,da)$.

\begin{center}
\begin{figure}[ht]
\caption{Fragments of the Lévy snake obtained from a mark $(s,a)$}\label{figure:fragments}

\scalebox{0.8} % Change this value to rescale the drawing.
{
\begin{pspicture}(0,-3.2167187)(14.799188,3.1937187)
\psline[linewidth=0.034cm,arrowsize=0.05291667cm 2.0,arrowlength=1.4,arrowinset=0.4]{->}(0.3021875,-2.3032813)(14.782187,-2.4232812)
\psline[linewidth=0.034cm,arrowsize=0.05291667cm 2.0,arrowlength=1.4,arrowinset=0.4]{->}(0.5221875,-2.6632812)(0.5221875,3.1767187)
\psline[linewidth=0.034cm,linestyle=dashed,dash=0.16cm 0.16cm](0.2821875,-0.30328125)(13.502188,-0.30328125)
\psline[linewidth=0.034](0.5421875,-2.2632813)(0.7021875,-1.6432812)(0.8221875,-1.9832813)(1.0621876,-1.1632812)(1.2221875,-1.6632812)(1.2621875,-1.4032812)(1.4221874,-1.8432813)(1.7421875,-0.62328124)(1.9021875,-1.0232812)(2.1621876,0.55671877)(2.3021874,-0.28328124)(2.4021876,0.09671875)(2.5221875,-0.30328125)(2.7421875,0.7967188)(2.9421875,-0.08328125)(3.1021874,0.43671876)(3.3021874,-0.30328125)(3.6021874,0.9167187)(3.9021876,-0.02328125)(4.0221877,0.21671875)
\psline[linewidth=0.034](4.0021877,0.21671875)(4.1621876,-0.28328124)(4.3221874,0.37671876)(4.5221877,-0.30328125)(4.6821876,0.23671874)(4.7821875,0.01671875)(5.1021876,1.0167187)(5.3021874,0.57671875)(5.3821874,0.7967188)(5.5821877,0.45671874)(5.9421873,1.7967187)(6.2821875,0.93671876)(6.3821874,1.2567188)(6.8021874,0.33671874)(6.9221873,0.61671877)(7.1621876,0.05671875)(7.3421874,0.81671876)(7.6621876,-0.14328125)(7.7621875,0.23671874)(8.002188,-0.32328126)(8.202188,0.35671875)(8.422188,-0.34328124)(8.602187,0.09671875)(8.722187,-0.14328125)(9.142187,1.0567187)(9.322187,0.71671873)(9.502188,1.3567188)(9.902187,0.19671875)(10.022187,0.5367187)(10.362187,-0.34328124)(10.542188,0.15671875)(10.982187,-1.0432812)(11.062187,-0.68328124)(11.282187,-1.2232813)(11.522187,0.01671875)(11.622188,-0.22328125)(11.782187,0.31671876)(12.162188,-1.3832812)(12.322187,-1.0032812)(12.522187,-1.5632813)(12.722187,-0.86328125)(12.982187,-1.8432813)(13.102187,-1.4432813)(13.442187,-2.3832812)(13.462188,-2.3632812)
\psline[linewidth=0.027999999cm,linestyle=dashed,dash=0.16cm 0.16cm](2.0221875,-0.30328125)(2.0621874,-2.3032813)
\psline[linewidth=0.027999999cm,linestyle=dashed,dash=0.16cm 0.16cm](4.5421877,-0.28328124)(4.5421877,-2.2632813)
\psline[linewidth=0.027999999cm,linestyle=dashed,dash=0.16cm 0.16cm](8.002188,-0.30328125)(8.042188,-2.2832813)
\psline[linewidth=0.027999999cm,linestyle=dashed,dash=0.16cm 0.16cm](10.7421875,-0.34328124)(10.802188,-2.3232813)
\psline[linewidth=0.027999999cm,linestyle=dashed,dash=0.16cm 0.16cm](6.1821876,2.2567186)(6.2621875,-3.0032814)
\psline[linewidth=0.034cm,tbarsize=0.07055555cm 5.0]{|-|}(2.0821874,-2.6632812)(4.5821877,-2.6632812)
\psline[linewidth=0.034cm,tbarsize=0.07055555cm 5.0]{|-|}(8.062187,-2.6432812)(10.782187,-2.6432812)
\usefont{T1}{ptm}{m}{n}
\rput(2.5596876,-2.0232813){\small $\alpha_{i_0}$}
\usefont{T1}{ptm}{m}{n}
\rput(5.1196876,-2.0432813){\small $\alpha_{i_s}$}
\usefont{T1}{ptm}{m}{n}
\rput(6.5496874,-2.0632813){\small $s$}
\usefont{T1}{ptm}{m}{n}
\rput(8.549687,-2.0832813){\small $\beta_{i_s}$}
\usefont{T1}{ptm}{m}{n}
\rput(11.229688,-2.1232812){\small $\beta_{i_0}$}
\usefont{T1}{ptm}{m}{n}
\rput(13.649688,-2.6432812){\small $\sigma$}
\usefont{T1}{ptm}{m}{n}
\rput(9.519688,-3.0032814){\small $\tilde{I}_+$}
\usefont{T1}{ptm}{m}{n}
\rput(3.2796874,-2.9832811){\small $\tilde{I}_-$}
\usefont{T1}{ptm}{m}{n}
\rput(0.2096875,-0.04328125){\small $a$}
\end{pspicture} 
}

\end{figure}
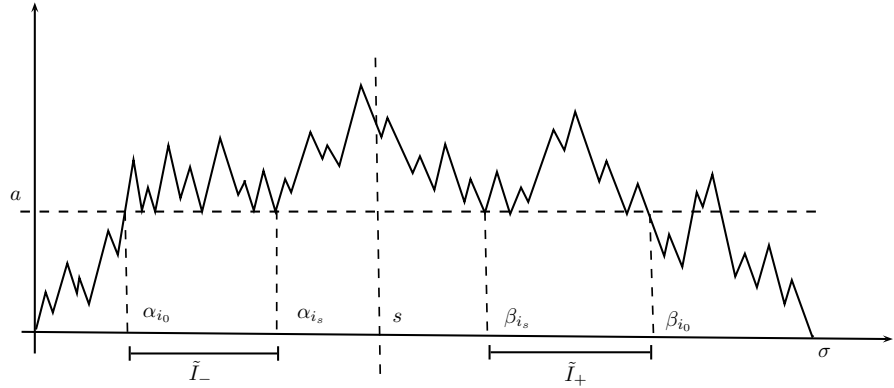
\end{center}

\section{The dislocation process}

Let $\ct$ be the set of jumping times of the Poisson process $\cq$. For $\theta\in\ct$, we consider the processes $\cl^{(\theta)}=(\rho^i ; i\in I^{(\theta)})$ and $\cl^{(\theta-)}=(\rho^i ; i\in I^{(\theta-)})$ defined in the Section \ref{relation}. The life times $(\sigma(\rho^i);i\in I^{(\theta)})$ (resp. $(\sigma(\rho^i);i\in I^{(\theta-)})$ ), ranked by decreasing order, of these Lévy snakes correspond to the "sizes" of the fragments at time $\theta$ (resp. before time $\theta$). Notice that, for $\theta\in\ct$ fixed, the families $\cl^{(\theta)}$ and $\cl^{(\theta-)}$ change in one family : the snake $\rho^{i_\theta}$ breaks in one family $(\rho^i, i\in \tilde{I}^{(\theta)})\in \cl^{(\theta)}$. Thus we get 
$$\cl^{(\theta)} = \left( \cl^{(\theta-)}\setminus \{ \rho^{i_\theta} \} \right) \cup \{ \rho^{i} ; i\in \tilde{I}^{(\theta)} \}.$$

%From the representation given in the previous Section, this fragmentation is given by the measure $q_\rho$, that is to say by the measure $\tilde{\nu}_\rho$. Thanks to  Lemma \ref{sjk}, 
%$$\sum_{\theta \in \ct }\delta_{(\theta, \cl^{(\theta-)}, (\rho^i,i\in \tilde{I}^{(\theta)}))}$$
%is a point process with intensity $d\theta ~ \delta_{\cl^{(\theta-)}} ~ \sum_{\rho \in \cl^{(\theta-)}} \tilde{\nu}_\rho$. 
%By taking the expectation over the snakes conditionally on their length, the process 
%$$\sum_{\theta \in \ct }\delta_{(\theta, (\sigma(\rho),\rho\in\cl^{(\theta-)}), (\sigma(\rho^i),i\in \tilde{I}^{(\theta)}))}$$\label{mesuredislocation}
%is a Poisson point process with intensity $d\theta ~ \delta_{(\sigma(\rho),\rho\in\cl^{(\theta-)})} ~ \sum_{\rho \in \cl^{(\theta-)}} \nu_{\sigma(\rho)}$

Let $\nu_{r}$ be the distribution of the decreasing lengths of Lévy snakes under $\tilde{\nu}_\rho$, integrated w.r.t. the law of $\rho$ conditionally on $\sigma(\rho)=r$, that is to say, for any non-negative measurable function $F$ defined on $\cs^\downarrow$
$$\int_{\cs^\downarrow}F(\mathbf{x}) \nu_r(d\mathbf{x}) = \N_r\left[\int F ((\sigma^i,i\in \tilde I)) \tilde \nu_\rho (d(\rho^i , i\in \tilde I))   \right]$$
where the $( \sigma^i , i\in \tilde I)$ are the lengths of the fragments $(\rho^i,  i\in \tilde I)$ ranked in decreasing order.

The family of measures $(\nu_r,r>0)$ is then the family of dislocation measures defined in the Section \ref{intro:section3}. Indeed, the formula above means that $\nu_r$ gives the distribution of the lengths of the fragments (ranked by decreasing order) coming from the dislocation of one fragment of size $r$. For $\mathbf{x}=(x_i,i\in I^{(\theta)})\in \cs^\downarrow$, if we consider the dislocation of all the fragments of $\mathbf{x}$ with respective sizes $x_i>0$, we found the formula of the dislocation measure $\nu_{x_i}$ given in Section \ref{intro:section3}:

$$\int F(\mathbf{s}) \tilde{\nu}_{\mathbf{x}}(d\mathbf{s}) = \sum_{i\geq1,x_i>0} \int F(\mathbf{x}^{i,\mathbf{s}})\nu_{x_i}(d\mathbf{s})$$

where $\tilde{\nu}_{\mathbf{x}}$ is defined as the intensity of a Poisson point process and is the law of the lengths $(x_i,i>0)$.

\subsection{Computation of dislocation measure}
We are interested in the family of dislocation measures $(\nu_r,r>0)$. Recall that $\N(.)=\int_{\R_+}\pi_*(dr)\N_r(.)$. The computation is easier under $\N$, then we compute for any $\lambda \geq 0$ :
\begin{eqnarray*}
\int_{\R_+ \times \cs^\downarrow} F (\mathbf{x})  \nu_r (d\mathbf{x})\pi_*(dr) &=& \N \left[  \int q_\rho (ds,da) F((\sigma^i,i\in \tilde I)) \right]\\
&=& \N \left[  \int q^{nod}_\rho (ds,da) F((\sigma^i,i\in \tilde I)) \right]+
\N \left[  \int q^{ske}_\rho (ds,da) F(\sigma^{i_0},\sigma^{i_s}) \right]
\end{eqnarray*}
where we use the decomposition of $q_\rho$ for the second equality. The first part has already been computed in \cite{ad:falp}. Jumping times of the process $\rho$ are represented by a subordinator $W$ with Laplace exponent $\psi'-\alpha$. Then we construct the length of the excursions of the snake by $S_W$ where $S$ is a subordinator with exponent $\psi^{-1}$, independent of $W$. Then we have :
$$ \N \left[  e^{-\lambda \sigma} \int q^{nod}_\rho (ds,da) F((\sigma^i,i\in \tilde I)) \right]= 
\int \pi(dv) \E \left[  S_v e^{-\lambda S_v} F\left( (\Delta S_u , u\leq v) \right)\right].$$
We now compute the second part. Thanks to the definition of the snake, $\rho^{ske}=0$ if and only if $\beta=0$ and in this case, we don't put mark on the skeleton of the tree. We assume that $\beta>0$ and we write the key lemma of this article which prove the Part 2 of the Theorem \ref{theointro}.

\begin{lem}\label{dislocation}
We set $\lambda_1>0$ and $\lambda_2> 0$.
$$\N \left[  \int q^{ske}_\rho (ds,da) \sigma^{i_s}e^{-\lambda_1\sigma^{i_s}-\lambda_2\sigma^{i_0}} \right] = \frac{2\beta }{\psi'\psi^{-1}(\lambda_1)\psi'\psi^{-1}(\lambda_2)}.$$
\end{lem}

We recall that the measure $\hat{\nu}_r^{ske}$ gives the law of the non-reordering of the two lengths given by the fragmentation from $\nu_r^{ske}$.

\begin{proof}[Proof of Part 2 of Theorem \ref{theointro}]
We use Lemma \ref{dislocation}, let $x_1$ and $x_2$ be the lengths of the fragments from $\nu_\rho^{ske}$ ranked by decreasing order among the elements of $x\in \cs^{\downarrow}$, we get
$$ \int_{\R_+\times \cs^{\downarrow}}x_1e^{-\lambda_1x_1-\lambda_2x_2}\hat{\nu}_r^{ske}(d\mathbf{x})\pi_*(dr)  = \frac{2\beta }{\psi'\psi^{-1}(\lambda_1)\psi'\psi^{-1}(\lambda_2)}.$$
We integrate w.r.t. $\lambda_1$ and we take the primitive which vanishes in 0, and we do the same with $\lambda_2$. We get that, for $\lambda_1>0$ and $\lambda_2>0$,
$$\int_{\R_+\times \cs^{\downarrow}} \frac{1}{x_2} (1-e^{-\lambda_1 x_1})(1-e^{-\lambda_2 x_2}) \hat{\nu}_r^{ske}(d\mathbf{x})\pi_*(dr) = 2\beta \psi^{-1}(\lambda_1)\psi^{-1}(\lambda_2).$$
Thus, under $\hat{\nu}_r^{ske}(d\mathbf{x})\pi_*(dr)$, the lengths of the two fragments are independent.
\end{proof}

\begin{proof}[Proof of Lemma \ref{dislocation}] In order to prove the lemma, we compute $A_2:=\N \left[  \int q^{ske}_\rho (ds,da) G(\sigma^{i_s},\sigma) \right]$ where $G(x,y)=xe^{-\lambda_1 x-\lambda_2 y}$. 
\begin{eqnarray*}
A_2 &=& \N \left[ 2\beta \int_0^\sigma ds \int \frac{1}{d_{s,a}-g_{s,a}} G(\sigma^{i_s},\sigma)\ind_{(0\leq a\leq H_s) }da \right]\\
&=&  \N \left[  2\beta \int_0^\sigma ds \int \frac{1}{d_{s,a}-g_{s,a}} G(d_{s,a}-g_{s,a},\sigma)\ind_{(0\leq a\leq H_s) }da \right].\\
\end{eqnarray*}

We denote for $0\leq s \leq \sigma$ and $0\leq a \leq H_s$ fixed
$$d_{s,a}-s=\inf\{t\geq 0, H_{(s+t)\wedge \sigma}\leq a\}=J_2(a)$$
$$s-g_{s,a}=\inf\{t\geq 0, H_{(s-t)_+}\leq a\}=J_1(a).$$
We get
$$d_{s,a}-g_{s,a}=J_2(a)+J_1(a),$$
$$\sigma = J_2(0)+J_1(0),$$

$$ A_2 = \N \left[  2\beta \int_0^\sigma ds \int\ind_{0\leq a\leq H_s }da  \frac{ G(J_1(a)+J_2(a),J_1(0)+J_2(0))}{J_1(a)+J_2(a)}\right].$$

We use the generalization for Lévy processes of Bismut formula, Proposition \ref{Bismut}.

\begin{eqnarray*}
A_2 &=& 2\beta \int \M(d\mu d\nu) \E\left[ \int \ind_{0\leq a \leq H(\mu) }da  \frac{ G(J^\nu(a)+J^\mu(a),J^\nu(0)+J^\mu(0))}{J^\nu(a)+J^\mu(a)}  \right]\\
&=& 2\beta\int \M(d\mu d\nu) \ind_{0\leq a \leq H(\mu) }da \, \E_\mu^* \left[ e^{-\lambda_1 J^\mu(a)-\lambda_2 J^\mu(0)} \right]\E_\nu^* \left[ e^{-\lambda_1 J^\nu(a)-\lambda_2 J^\nu(0)} \right]
\end{eqnarray*}

where $J^\mu(a)$ is the first passage time of the process $H^{(\mu)}$ at level $a$. By the Poissonnian decomposition of $\rho$ under $\P_\mu^*$ w.r.t. the excursions of $\rho$ above its minimum, under $\P_\mu^*$, we replace respectively $J^\mu(0)$ and $J^{\mu}(a)$ by ${\sum_{i\in I}}\sigma^i$ and ${\sum_{h_i\geq a}}\sigma^i$. We separate ${\sum_{i\in I}}\sigma^i={\sum_{h_i\geq a}}\sigma^i + {\sum_{h_i<a}}\sigma^i$.\\

$\displaystyle A_2 = 2\beta \int \M(d\mu d\nu)\int \ind_{0\leq a \leq H(\mu) }da \E^*_{\mu}\left[ exp\left(-(\lambda_1+\lambda_2) {\sum_{h_i \geq a}}\sigma^i -\lambda_2{\sum_{h_i < a}}\sigma^i  \right) \right]$\\
\begin{flushright}
$\displaystyle \E^*_{\nu}\left[ exp\left(-(\lambda_1+\lambda_2) {\sum_{h_i \geq a}}\sigma^i -\lambda_2{\sum_{h_i < a}}\sigma^i  \right) \right].$\\
\end{flushright}

Using standard properties of Poisson point measures, the atoms above level $a$ are independent of the atoms below, the expectations can be separated.

$$\E^*_{\mu}\left[ e^{-(\lambda_1+\lambda_2) {\sum_{h_i\geq a}}\sigma^i -\lambda_2{\sum_{h_i < a}}\sigma^i } \right] = 
\E^*_{\mu}\left[ e^{-(\lambda_1+\lambda_2) {\sum_{h_i\geq a}}\sigma^i   }\right] \E^*_{\mu}\left[ e^{-\lambda_2{\sum_{h_i<a}}\sigma^i}\right].$$

We use Lemma \ref{decomp poisson}, and the equality $\psi^{-1}(\lambda)=\N\left[1-e^{-\lambda \sigma} \right]$, we get 
$$\E^*_{\mu}\left[ e^{-(\lambda_1+\lambda_2) {\sum_{h_i\geq a}}\sigma^i   }\right]=e^{-\mu([a,H(\mu)])\N\left[1-e^{-(\lambda_1+\lambda_2)\sigma}  \right]} = e^{-\mu([a,b])\psi^{-1}(\lambda_1+\lambda_2)}.$$
And we do the same for the second expectation.

\begin{eqnarray*}
A_2 &=& 2\beta  \int_0^\infty db e^{-\alpha b}\int_0^b da\M_b\left[ e^{-(\mu+\nu)([a,b])\psi^{-1}(\lambda_1+\lambda_2) } e^{-(\mu+\nu)([0,a))\psi^{-1}(\lambda_2)}\right].\\
\end{eqnarray*}
Then,\\
\begin{tabular}{rcl}
$\displaystyle \M_b\left[ e^{-((\mu+\nu)([a,b])\psi^{-1}(\lambda_1+\lambda_2) } e^{-(\mu+\nu)([0,a))\psi^{-1}(\lambda_2)}\right] $&&\\
\multicolumn{3}{l}{\hspace{1cm}$=\displaystyle \M_b\left[ e^{-((\mu+\nu)([a,b])\psi^{-1}(\lambda_1+\lambda_2) }\right]\M_b \left[ e^{-(\mu+\nu)([0,a))\psi^{-1}(\lambda_2)} \right]$}\\
\multicolumn{3}{l}{\hspace{1cm}= $\displaystyle e^{-2(b-a)\beta\psi^{-1}(\lambda_1+\lambda_2)}  exp\left(-\int_a^b dx \int_0^\infty l\pi(dl) (1-e^{-l\psi^{-1}(\lambda_1+\lambda_2)}) \right)$ }\\
\multicolumn{3}{l}{\hspace{1cm} $\displaystyle e^{ -2a\beta\psi^{-1}(\lambda_2)} \displaystyle exp\left(-\int_0^a dx \int_0^\infty l\pi(dl) (1-e^{-l\psi^{-1}(\lambda_2)}) \right)$ }\\
\multicolumn{3}{l}{ \hspace{1cm}$= \displaystyle e^{\alpha b} e^{-(b-a)\psi'\psi^{-1}(\lambda_1+\lambda_2)-a\psi'\psi^{-1}(\lambda_2)} .$}
\end{tabular}

We recall the expression of $A_2$
\begin{eqnarray*}
A_2 &=& 2\beta \int_0^\infty db \frac{e^{-b\psi'\psi^{-1}(\lambda_2)}-e^{-b\psi'\psi^{-1}(\lambda_1+\lambda_2)}}{\psi'\psi^{-1}(\lambda_1+\lambda_2)-\psi'\psi^{-1}(\lambda_2)}\\
&=& \frac{2\beta}{\psi'\psi^{-1}(\lambda_1+\lambda_2)-\psi'\psi^{-1}(\lambda_2)}\left( \frac{1}{\psi'\psi^{-1}(\lambda_2)}-\frac{1}{\psi'\psi^{-1}(\lambda_1+\lambda_2)}  \right)\\
&=& \frac{2\beta }{\psi'\psi^{-1}(\lambda_2)\psi'\psi^{-1}(\lambda_1+\lambda_2)}.
\end{eqnarray*}

We use the equality $\N \left[  \int q^{ske}_\rho (ds,da)\sigma^{i_s}G(\sigma^{i_s},\sigma)\right] = \N \left[  \int q^{ske}_\rho (ds,da)\sigma^{i_s}e^{-(\lambda_1+\lambda_2)\sigma^{i_s}-\lambda_2\sigma^{i_0}}\right]$, we finally get the result.
\end{proof}

\subsection{Brownian case}

A similar result has been obtained by Abraham and Serlet \cite{as:psf} in the Brownian case and conditionally on $\sigma=1$. They use the same construction of the marks on the skeleton given by Aldous and Pitman \cite{ap:sac}. 

We consider a standard Brownian motion with Laplace exponent $\psi(\lambda)= \frac{\lambda^2}{2}$ and we denote by $\Gamma(de)$ the law of the Brownian excursion $e$. Thanks to \cite{lg:sbprspde}, Section VIII.3, the height process of the Brownian motion is given by $H_t = 2(X_t-I_t)$. We resume the computation of \cite{as:psf} by taking marks under the epigraph  of $H$, we get 
$$\int F(\sigma^{i_s},\sigma)\nu(ds) = \int \Gamma(de) \int_0^\sigma ds \int_0^{2e(s)} dt \frac{F(\sigma^{i_s},\sigma)}{\sigma^{i_s}}$$
where $\nu$ is the dislocation measure of \cite{as:psf}. The computation of \cite{as:psf} uses the law the two independent 3-dimensional Bessel processes, then we get
$$\int F(\sigma^{i_s},\sigma)\nu(ds) = \frac{1}{4\pi} \int_0^1 \frac{dz}{\sqrt{z(1-z)}} \int_0^\infty d\sigma \frac{F(\sigma z,\sigma)}{\sigma z}.$$
As before, we compute with $F(x,y)=xe^{-\lambda_1 x - \lambda_2 y}$.
\begin{eqnarray*}
\int F(\sigma^{i_s},\sigma)\nu(ds) &=& \frac{1}{4\pi} \int_0^1 \frac{dz}{\sqrt{z(1-z)}} \int_0^\infty d\sigma e^{-\lambda_1 \sigma z - \lambda_2 \sigma}\\
&=& \frac{1}{4\pi} \int_0^1 \frac{dz}{\sqrt{z(1-z)}} \frac{1}{\lambda_1 z + \lambda_2}
\end{eqnarray*}
For the end of this computation, we use the two changes of variable : $z\leftrightarrow sin^2 x$ and then $t\leftrightarrow tan x$.
\begin{eqnarray*}
\int F(\sigma^{i_s},\sigma)\nu(ds) &=& \frac{1}{2\pi} \int_0^{\frac{\pi}{2}} \frac{dx}{\lambda_1 sin^2 x + \lambda_2}\\
&=& \frac{1}{2\pi} \int_0^\infty \frac{dt}{(\lambda_1+\lambda_2)t^2+\lambda_2}
\end{eqnarray*}

We integrate a last time, we get the same result as in Lemma \ref{dislocation} :
$$\int F(\sigma^{i_s},\sigma)\nu(ds) = \frac{1}{4}\frac{1}{\sqrt{\lambda_2(\lambda_1+\lambda_2)}}.$$
\\

\begin{center}
\scshape{Acknowledgement}
\end{center}
I am grateful to my PhD. advisor Romain Abraham for his helpful discussions.

\end{document}